\pdfoutput=1
\documentclass{article}
\usepackage{graphicx,hyperref}
\usepackage{amsmath,amsfonts,amssymb,latexsym,amsthm}
\usepackage{multicol,multirow}
\usepackage{placeins}

\usepackage{bm}
\providecommand{\B}{}
\renewcommand{\B}{\bm}
\newcommand{\D}{\partial}
\renewcommand{\le}{\leqslant}
\renewcommand{\ge}{\geqslant}
\newtheorem{theorem}{Theorem}
\newtheorem{lemma}[theorem]{Lemma}

\theoremstyle{definition}

\usepackage{multirow}
\newcommand{\He}{H\!e}
\newcommand{\J}{\mathcal{J}}

\begin{document}

\title{Computation of the solution\\ for the 2D acoustic pulse propagation}

\author{P.A.~Bakhvalov}

\date{April 16, 2024}

\maketitle

\begin{abstract}
We consider the 2D acoustic system with the Gaussian pulse as the initial data. This case was proposed at the first Workshop on benchmark problems in computational aeroacoustics, and it is commonly used for the verification of numerical methods. We construct an efficient algorithm to evaluate the exact solution for a given time $t$ and distance $r$. For a precision $\varepsilon$, it takes $c\ln(1/\varepsilon)$ operations (the evaluation of a Bessel function counts as one operation) where $c$ does not depend on $t$ and $r$. This becomes possible by using three different integral representations and an asymptotic series depending on $t$ and $r$.
\end{abstract}

\medskip

\sloppy

\section{Introduction}

The verification of a computational method or its program implementation implies the comparison of the numerical results with some reference data. Verification usually goes along with the progressive increase of the complexity and begins with simple cases \cite{AIAAGuide,Masatsuka}. At the first stage, it is favourable if the exact solution is available at any point of the computational domain. So cases with closed-form exact solutions are widely used. Integral form solutions are equally convenient provided that they may be computed precisely and efficiently.

The computational costs of the exact solution evaluation become more important for high-order finite-volume methods. In these schemes, one compares the numerical solution at a mesh cell with the cell average of the exact solution. Thus, to evaluate the average at one cell, one needs to evaluate the exact solution at several points, depending on the order of the scheme. A similar difficulty presents in some finite-element schemes where the numerical solutions are compared to the $L_2$-projection of the exact solution.

The first Workshop on benchmark problems in computational aeroacoustics \cite{CAA1995}, which held in 1995, considered several verification cases. Cases 3.1 and 3.2 in \cite{CAA1995} were the Cauchy problems for the linearized Euler equations with a uniform background flow. A general solution of these equations is a combination of acoustic, vortex, and entropy components, see \cite{Tam1993}. In the coordinate system of the background flow, the vortex and entropy components are steady, and the acoustic one satisfies the acoustic system. In particular, for the tests in \cite{CAA1995}, we come to the Cauchy problem 
\begin{equation}
\frac{\D p'}{\D t} + \nabla \cdot \B{u}' = 0,
\quad
\frac{\D \B{u}'}{\D t} + \nabla p' = 0,
\label{eq:acoustic1}
\end{equation}
in $\mathbb{R}^2$ with the initial data
\begin{equation}
p'(0,\B{r}) = \exp(-\B{r}^2/2), \quad \B{u}'(0,\B{r}) = 0.
\label{eq:acoustic2}
\end{equation}
Here $p'$ and $\B{u}'$ are the pressure and velocity pulsations, correspondingly. This benchmark problem is widely used \cite{Sescu2008, Ramirez2018, Zhang2019, Song2022}.

The solution of \eqref{eq:acoustic1}--\eqref{eq:acoustic2} is given by 
\begin{equation}
\left(\begin{array}{c}p'(t,r) \\ u_r'(t,r) \end{array}\right) = \int\limits_0^{\infty} \omega e^{-\omega^2/2} 
\left(\begin{array}{c}J_0(r\omega) \cos(t\omega) \\ J_1(r\omega) \sin(t\omega) \end{array}\right) d\omega,
\label{eq_form1}
\end{equation}
where $J_j(x)$ are the Bessel functions of the first kind and index $j$.
In \cite{CAA1995, Tam1993} it was suggested to apply a quadrature to \eqref{eq_form1}. For $t, r \le 20$, Gauss quadratures are fast and accurate. However, for bigger $t$ or $r$ they become inefficient.

Quadratures for integrals involving $J_{\nu}(r\omega)$, $r \gg 1$, were constructed in \mbox{\cite{Asheim2013, Chen2018, Kang2019}}, see also references therein. These methods are connected with the ones for integrals of the form $\int f(x) e^{i \omega g(x)}dx$, $\omega \gg 1$, which may be found in \cite{Huybrechs2006, Gao2017}. The expression \eqref{eq_form1} contains products of two different oscillatory kernels; the evaluation of such integrals is studied in \cite{Li2020, Khan2023}. All these methods are based on the continuation of the integrand to the complex domain and the proper choice of the integration path. They are well capable to evaluate the integral \eqref{eq_form1}, but their computational costs is high. For instance, the evaluation of one oscillatory integral similar to \eqref{eq_form1} in \cite{Li2020} takes $0.016$ seconds. The exact form of \eqref{eq_form1} does not help to simplify these methods, so for our purpose they take too much CPU time.

In this paper we are concerned with the algorithm to evaluate the solution of \eqref{eq:acoustic1}--\eqref{eq:acoustic2} for given $t,r \ge 0$ with a requested absolute precision $\varepsilon$. The algorithm is designed for this particular problem and cannot be easily generalized to evaluate any other integral. In other words, we consider $p'(t,r)$ and $u_r'(t,r)$ defined by \eqref{eq_form1} as special functions and want to approximate it. Our method takes approximately $10^{-6}$ seconds in the double precision and $10^{-4}$ seconds in the double-double precision.

The idea of the algorithm is to use integral transforms and thus avoid the approximation of oscillatory integrals. The algorithm takes at most $c\ln (1/\varepsilon)$ operations with $c$ independent of $t$ and $r$, provided that the evaluation of a Bessel function counts as one operation. We do not intend to get the relative precision $\varepsilon$, but this may be achieved by a small modification of our method. We also take precautions to avoid accumulating the arithmetic errors. The algorithm is implemented in the open-source ColESo library (Collection of Exact Solutions for verification of numerical algorithms for simulation of compressible flows) \cite{Bakhvalov2022ColESo}. It is verified for the double precision and the double-double precision arithmetics.

The rest of the paper is structured as follows. Section~\ref{sect:overview} gives an overview of the algorithm. In Section~\ref{sect:asympt} we present the asymptotic series as $t \rightarrow \infty$ for small $r$. In Section~\ref{sect:solution} we prove the integral expressions we use. In Section~\ref{sect:accuracy} we study the accuracy of quadrature rules. Section~\ref{sect:experiment} contains the results of the actual accuracy (i. e. in a finite precision arithmetics) and the computational time.

\section{Overview of the algorithm}
\label{sect:overview}

Depending on $t$ and $r$, we use three integral forms of the solution or an asymptotic series. In this section we present them and describe which method we use for $t$ and $r$ given. The proofs of these forms are left to the following sections. 

The first form is given by \eqref{eq_form1}.

The second form is
\begin{equation}
p'(t,r) = \J_0(t,r) + \J_0(-t,r), \quad
u_r'(t,r) = \J_1(t,r) - \J_1(-t,r),
\label{eq_tamnew_2r}
\end{equation}
\begin{equation}
\J_j(t,r) = \frac{1}{\sqrt{2\pi}} \int\limits_0^{\infty} \frac{\exp(-(r-t+r\xi)^2/2)}{\sqrt{\xi(\xi+2)}} (r-t+r\xi) (1+\xi)^j d\xi.
\label{eq_tamnew_3r}
\end{equation}
For $r=0$, this integral diverges for both $j=0$ and $j=1$, so this form is not applicable. 

The third form is 
\begin{equation}
\begin{gathered}
p'(t,r) = \J_{0,1}(t,r) - t^2 \J_{0,3}(t,r) + rt \J_{1,2}(t,r),
\\
u_r'(t,r)  =  -t^2 \J_{1,2}(t,r) + rt \J_{0,1}(t,r),
\end{gathered}
\label{eq_puls_bf}
\end{equation}
\begin{equation}
\J_{j,n}(t,r) = 
\int\limits_0^{1} \exp\left(-\frac{(r-t+t\xi)^2}{2}\right) 
\tilde{I}_j\left(rt(1-\xi)\right) \frac{(1-\xi)^n}{\sqrt{\xi} \sqrt{2 - \xi}} d\xi.
\label{eq_jjn}
\end{equation}
Here $\tilde{I}_j(x) = e^{-x} I_j(x)$, and $I_j(x)$ is the modified Bessel function of the first kind and index $j$.

We also use the asymptotic series as $t \rightarrow \infty$ together with the Taylor series at $r=0$. The details are presented in the next section.

Let us explain the necessity of all the representations. The integrand of \eqref{eq_form1} becomes an oscillatory function for large $t$ or $r$ making the use of this form ineffective. In reverse, the integrand of \eqref{eq_tamnew_3r} have a singularity at $\xi = -2$, which for small $t$ and $r$ deteriorates the integration accuracy (although the accumulation of the arithmetic errors due to $\vert\J_j(t,r)\vert \rightarrow \infty$ as $r \rightarrow 0$ can be avoided). The form \eqref{eq_puls_bf}--\eqref{eq_jjn} is applicable for all $t$ and $r$, however, we want to avoid the computation of  $\tilde{I}_j(x)$ for $2 < x \lesssim H^2$.

Now we define which form and which quadrature (uniform-step quadrature, the Gauss~-- Legendre quadrature, and the Gauss~-- Jacobi quadrature) is used for $t$ and $r$ given. Applying a quadrature, we crop the integration range to the interval where the integrand is essentially nonzero. Put 
\begin{equation}
H = \sqrt{-2 \ln(\varepsilon/2)}, \quad R_1 = (7.5\varepsilon)^{1/6}, \quad R_2 = 5\varepsilon^{1/10}.
\label{eq_def_H}
\end{equation}
The domains in $(t,r)$ corresponding to different methods are shown in Fig.~\ref{fig:variants}. 

\begin{figure}[t]
\begin{center}
\includegraphics[width=0.5\linewidth]{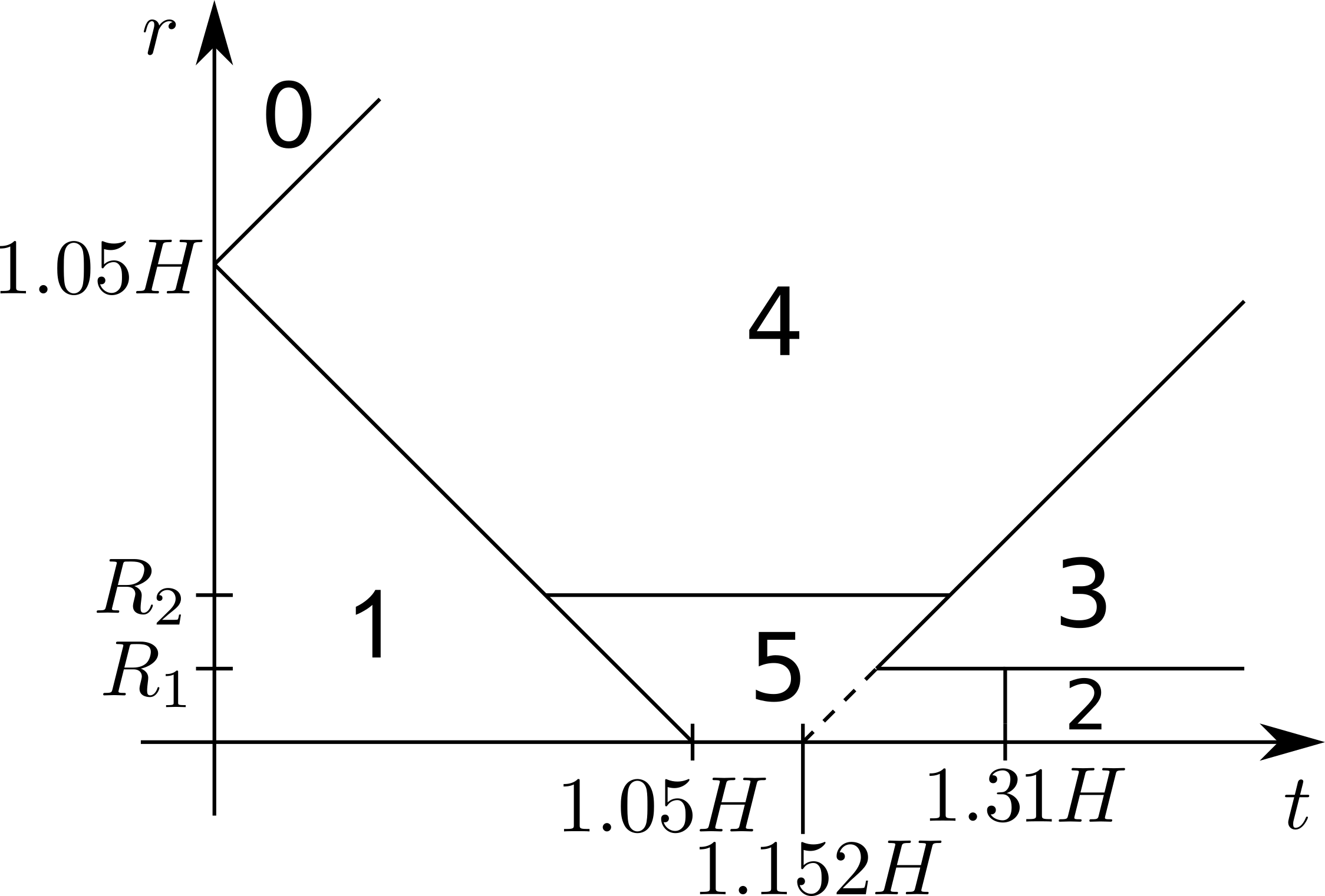}
\end{center}
\caption{Solution form used depending on $t$ and $r$. 0: $p',u_r'\approx 0$; 1: \eqref{eq_form1}; 2: series; 3, 4: \eqref{eq_tamnew_2r}--\eqref{eq_tamnew_3r}; 5: \eqref{eq_puls_bf}--\eqref{eq_jjn}}
\label{fig:variants}
\end{figure}

\begin{enumerate}
\item If $t-r > 1.152 H$, then:
\begin{enumerate}
\item if $r > R_1$, use the representation \eqref{eq_tamnew_2r}--\eqref{eq_tamnew_3r} and the uniform-step quadrature (zone 3).
\item if $r \le R_1$ and $t \ge 1.31H$, use the approximation from Section~\ref{sect:asympt} (zone 2).
\item if $r \le R_1$ and $t < 1.31H$, use the representation \eqref{eq_puls_bf}--\eqref{eq_jjn} and the Gauss~-- Legendre quadrature (zone 5).
\end{enumerate}
\item If $t-r \le 1.152 H$, then:
\begin{enumerate}
\item if $t < \varepsilon$, put $p'(t,r) = \exp(-r^2/2)$ and $u_r'(t,r) = -tr\exp(-r^2/2)$ (not shown on Fig.~\ref{fig:variants});
\item else if $t < r-1.05 H$, put $p'(t,r) = u_r'(t,r)=0$ (zone 0);
\item else if $t+r < 1.05 H$, use the representation \eqref{eq_form1} and the Gauss~-- Legendre quadrature formulas (zone 1);
\item else if $t+r \ge 1.05 H$ and $r \le R_2$, use the representation \eqref{eq_puls_bf}--\eqref{eq_jjn} and the Gauss~-- Legendre quadrature formulas (zone 5);
\item else use the representation \eqref{eq_tamnew_2r}--\eqref{eq_tamnew_3r} and the Gauss~-- Jacobi quadrature formulas (zone 4).
\end{enumerate}
\end{enumerate}

For the uniform-step quadrature rule, the number of nodes is $2M_2+1$ with $M_2 = \lceil 0.2 H^2 \rceil$. For the Gauss rules, the number of nodes is $M_3 = \lceil 0.71 H^2 \rceil$. Recalling \eqref{eq_def_H} we see that the number of quadrature nodes is $O(\ln(1/\varepsilon))$ and so is the total number of arithmetic operations provided that the evaluation of the integrands in \eqref{eq_form1}, \eqref{eq_tamnew_3r}, \eqref{eq_jjn} at one point takes $O(1)$ operations.

The constants $M_2$, $M_3$, $R_1$, $R_2$, $1.05 H$, $1.152 H$, $1.31 H$ (see Fig.~\ref{fig:variants}) are taken to obtain the result with the accuracy $\varepsilon$ provided that $\varepsilon \le 2 \cdot 10^{-16}$. To get a faster and less accurate algorithm, these constants should be calibrated for a specific $\varepsilon$.

The nodes and weights of the Gauss rules are computed by the Golub~-- Welsh algorithm \cite{Golub1968} as implemented in \cite{Elhay1987}. We assume them to be precomputed and therefore skip their construction when counting the number of arithmetic operations.

\section{Series as $r \rightarrow 0$, $t \rightarrow \infty$}
\label{sect:asympt}

Throughout this section, $\Gamma$ is the Euler gamma function.
Let $\He_n(\omega)$ be the Hermite polynomials:
$$
\He_n(\omega) = (-1)^n e^{\omega^2/2} \frac{d^n}{d\omega^n} e^{-\omega^2/2}.
$$
Since
$$
J_n(x) = \frac{1}{\pi} \int\limits_0^{\pi} \cos(n\phi - x\sin\phi)d\phi,
$$
for each $n, m \in \mathbb{N} \cup \{0\}$ there holds $\sup_x \vert d^m J_n/dx^m \vert \le 1$. Using the Taylor expansion with the Lagrange remainder, we get
\begin{equation*}
\begin{gathered}
\omega J_0(r\omega) = \left(\frac{15}{64}r^4 - \frac{3}{4}r^2 + 1\right) \He_1(\omega) + \left(\frac{5}{32} r^4 - \frac{1}{4}r^2\right)\He_3(\omega) + 
\\
+ \frac{1}{64}r^4 \He_5(\omega) + c_0(r\omega) \frac{\omega^7 r^6}{720},
\\
\omega J_1(r\omega) = \left(\frac{5}{128}r^5 - \frac{3}{16}r^3 + \frac{1}{2}r\right) \He_0(\omega) + \left(\frac{15}{128}r^5 - \frac{3}{8}r^3 + \frac{1}{2}r\right) \He_2(\omega) + 
\\
+ \left(\frac{5}{128} r^5 - \frac{1}{16}r^3\right)\He_4(\omega) + 
\frac{1}{384}r^5 \He_6(\omega) + c_1(r\omega) \frac{\omega^8 r^7}{5040}
\end{gathered}
\end{equation*}
where $\vert c_0(r\omega) \vert \le 1$, $\vert c_1(r\omega) \vert \le 1$.
Then \eqref{eq_form1} yields
\begin{equation*}
\begin{gathered}
p'(t,r) = \mathrm{Re}\left[\left(\frac{15}{64}r^4 - \frac{3}{4}r^2 + 1\right) \mathcal{I}_1(t) + \left(\frac{5}{32} r^4 - \frac{1}{4}r^2\right)\mathcal{I}_3(t) + \right.
\\
+ \left.\frac{1}{64}r^4 \mathcal{I}_5(t) \right] + \mathrm{Re} E_0,
\\
u_r'(t,r) = \mathrm{Im}\left[\left(\frac{5}{128}r^5 - \frac{1}{16}r^3 + \frac{1}{2}r\right) \mathcal{I}_0(t) + \left(\frac{15}{128}r^5 - \frac{3}{8}r^3 + \frac{1}{2}r\right) \mathcal{I}_2(t) + \right.
\\
\left.
+ \left(\frac{5}{128} r^5 - \frac{1}{16}r^3\right) \mathcal{I}_4(t) + 
\frac{1}{384}r^5 \mathcal{I}_6(t)\right] + \mathrm{Im} E_1
\end{gathered}
\end{equation*}
with
\begin{equation}
\mathcal{I}_n(t) = \int\limits_0^{\infty} \He_n(\omega) e^{-\omega^2/2} \exp(it\omega) d\omega
\label{eq_def_In}
\end{equation}
and
$$
E_j = \int\limits_0^{\infty} \omega e^{-\omega^2/2} \exp(it\omega) c_j(r\omega) \frac{(r \omega)^{6+j}}{(6+j)!} d\omega, \quad j = 0,1.
$$
From here,
$$
\vert E_j \vert \le \int\limits_0^{\infty} e^{-\omega^2/2} \omega^{7+j} d\omega \frac{r^{6+j} }{(6+j)!} = \frac{r^{6+j} }{(6+j)!} 2^{3+j/2} \Gamma(4+j/2).
$$
Note that $\Gamma(4)=6$, $\Gamma(4.5)=135\sqrt{\pi}/16 < 15$. Since this form is used for $r \le R_1 = (7.5 \varepsilon)^{1/6}$, we have $\vert E_0 \vert \le \varepsilon/2$ and $\vert E_1 \vert \le \varepsilon/2$. So we drop the terms $\mathrm{Re}E_0$ and $\mathrm{Im}E_1$ in this form.

It remains to show how to evaluate $\mathcal{I}_n(t)$.

\begin{lemma}\label{th:lemma:In}
There holds $\vert\mathcal{I}_n(t)\vert \le \sqrt{\pi/2}\ 2^{n/2}\ \Gamma(1+n/2)$.
\end{lemma}
\begin{proof}
Using the representation
$$
\He_n(x) = \frac{1}{\sqrt{2\pi}} \int\limits_{-\infty}^{\infty} (x+iy)^n \exp(-y^2/2) dy
$$
we get
$$
\vert\mathcal{I}_n(t)\vert \le \frac{1}{2}\int\limits_{-\infty}^{\infty} \vert\He_n(x)\vert \exp(-x^2/2) dx \le 
$$
$$
\le \frac{1}{2\sqrt{2\pi}} \int\limits_{-\infty}^{\infty} \int\limits_{-\infty}^{\infty} \vert x^2 + y^2\vert^{n/2} \exp\left(-\frac{x^2+y^2}{2}\right) dydx 
=
$$
$$
= \frac{2\pi}{2\sqrt{2\pi}} \int\limits_0^{\infty} r^{n+1} \exp(-r^2/2) dr =
\frac{2\pi}{2\sqrt{2\pi}} \int\limits_0^{\infty} (\sqrt{2x})^{n} \exp(-x) dx.
$$
Now the inequality to prove is by the definition of the Euler gamma function.
\end{proof}

For $\mathcal{I}_n(t)$, we consider the approximation
\begin{equation}
\mathcal{I}_n(t) \approx - i^{n-1} \sum\limits_{l=\lceil n/2 \rceil}^{\lfloor (M-1)/2 \rfloor} \frac{(2l-1)!!}{t^{2l-n+1}}.
\label{eq_as_In_intro}
\end{equation}
with $M = \lfloor H^2 \rfloor$. Taking $M = \infty$ in \eqref{eq_as_In_intro} makes it an asymptotic series, as it follows from the next lemma.

\begin{lemma}
Let $0 \le n < M < \infty$ and
$$
R_{n,M} = \mathcal{I}_n(t) + i^{n-1} \sum\limits_{l=\lceil n/2 \rceil}^{\lfloor (M-1)/2 \rfloor} \frac{(2l-1)!!}{t^{2l-n+1}}.
$$
Here we imply $(-1)!!=1$. Then there holds
\begin{equation}
\vert R_{n,M}(t) \vert \le \sqrt{\pi/2} \frac{1}{t^{M-n}} 2^{M/2}  \Gamma(1+M/2).
\label{eq_as_aux101}
\end{equation}
\end{lemma}

\begin{proof}
Integration by parts in \eqref{eq_def_In} gives
$$
\mathcal{I}_n(t) = \frac{1}{it} \left. \He_n(\omega) e^{-\omega^2/2} \exp(it\omega)\right\vert_{0}^{\infty} - \frac{1}{it} \int\limits_0^{\infty} \frac{d}{d\omega}\left(\He_n(\omega) e^{-\omega^2/2}\right) \exp(it\omega) d\omega 
=
$$
$$
=
\frac{-1}{it} \He_n(0) - \frac{-1}{it} \int\limits_0^{\infty} \He_{n+1}(\omega) e^{-\omega^2/2}  \exp(it\omega) d\omega = \frac{1}{it}(\mathcal{I}_{n+1}(t) - \He_n(0)).
$$
Hence,
$$
\mathcal{I}_n(t) = \frac{1}{(it)^{M-n}} \mathcal{I}_M(t) - \sum\limits_{k=n}^{M-1} \frac{1}{(it)^{k-n+1}} \He_k(0).
$$
There holds $\He_k(0) = 0$ for an odd $k$ and $\He_k(0) = (-1)^{k/2} (k-1)!!$ for an even $k$. Therefore,
$$
\mathcal{I}_n(t) = \frac{1}{(it)^{M-n}} \mathcal{I}_M(t) - i^{n-1} \sum\limits_{l=\lceil n/2 \rceil}^{\lfloor (M-1)/2 \rfloor} \frac{(2l-1)!!}{t^{2l-n+1}}
$$
and
$$
R_{n,M}(t) = (it)^{n-M} \mathcal{I}_M(t).
$$
It remains to use the estimate given by Lemma~\ref{th:lemma:In}.
\end{proof}


As stated in Section~\ref{sect:overview}, we use this series for $t \ge 1.31 H$ and $n \le 6$. Under these conditions, for $M = \lfloor H^2 \rfloor$, estimate \eqref{eq_as_aux101} yields $\vert R_{n,M}(t) \vert  \le \varepsilon/2$ provided that $\varepsilon \le 2\cdot 10^{-16}$.

Note that if $t \gg H$, then the use of $\lfloor H^2 /2 \rfloor $ terms is excessive. If a term in the sum \eqref{eq_as_In_intro} is small enough, then the rest of the terms may be dropped.

\section{Integral expressions}
\label{sect:solution}

In this section we prove three integral representation declared in Section~\ref{sect:overview}.

Introducing the wave potential $W(t,r)$, $r = \vert \B{r}\vert$, we come to the Cauchy problem
\begin{equation}
\frac{\D^2 W}{\D t^2} - \frac{1}{r}\frac{\D}{\D r}\left(r \frac{\D W}{\D r}\right) = 0, \quad t>0, \quad r>0,
\label{eq:intro:7}
\end{equation}
\begin{equation}
W\vert_{t=0} = 0, \quad \left.\frac{\D W}{\D t}\right\vert_{t=0} = -\exp\left(-\frac{r^2}{2}\right),
\label{eq:intro:id}
\end{equation}
and the pressure and velocity pulsations are given by
\begin{equation}
p'(t,\B{r}) = -\frac{\D W}{\D t}(t,r),\quad \B{u}'(t,\B{r}) = u_r'(t,r) \frac{\B{r}}{r}, \quad u_r'(t,r) = \frac{\D W}{\D r}(t,r).
\label{eq_def_pur}
\end{equation}

First, by the Fourier method, from \eqref{eq:intro:7}--\eqref{eq:intro:id} we have
\begin{equation}
W(t,r) = - \int\limits_0^{\infty} \exp\left(-\frac{\omega^2}{2}\right) \sin(\omega t) J_0(r \omega) d\omega
\label{eq_expr_W}
\end{equation}
where $J_0$ is the Bessel function of the first kind and zero index. Taking the time and radial derivatives (recall that $J_0' = J_1$) we come to \eqref{eq_form1}.

Recall the Parseval identity, which is also a definition of the Fourier transform of a distribution.
\begin{lemma}
For each $f \in S(\mathbb{R})$ and $g \in S'(\mathbb{R})$ there holds
$(f, g) = (F, G)$, where $F$ and $G$ are the Fourier images of $f$ and $g$, correspondingly.
\label{th:parseval}
\end{lemma}

\begin{lemma}
For each $t \ge 0$, $r > 0$, the formulas \eqref{eq_tamnew_2r}--\eqref{eq_tamnew_3r} give a solution of \eqref{eq:intro:7}--\eqref{eq_def_pur}.
\end{lemma}
\begin{proof}
Rewrite \eqref{eq_expr_W} as
\begin{equation}
W(t,r) = -\mathrm{Im}\,\int\limits_{-\infty}^{\infty} 
\exp\left(-\frac{\omega^2}{2}\right)
J_0(r\omega) \exp(i\omega t) \Theta(\omega) d\omega,
\label{eq_Lessel_sinus222}
\end{equation}
where $\Theta$ is the Heaviside function. There holds
\begin{equation}
\frac{1}{\sqrt{2\pi}} 
\int\limits_{-\infty}^{\infty}
\exp\left(-\frac{\omega^2}{2}\right) \exp(i k \omega) d\omega =
\exp\left(-\frac{k^2}{2}\right)
\label{eq_Lessel_sinus111}
\end{equation}
and
\begin{equation}
\frac{1}{\sqrt{2\pi}} 
\int\limits_{0}^{\infty} J_0(r\omega) \exp(i\omega (t+k))  d\omega 
= \frac{1}{\sqrt{2\pi}\sqrt{r^2 - (t+k)^2}},
\label{eq_Lessel_sinus}
\end{equation}
where we take the positive imaginary part if $t+k>r$ and the negative imaginary part if $t+k<-r$. Identity \eqref{eq_Lessel_sinus} for $r \ne \vert t+k \vert$ can be found in \cite{Gradstein1963eng}, Sect. 6.67. However, if we consider its left-hand side as the Fourier transform of $J_0(r\omega) \exp(i\omega t) \Theta(\omega)$, it also holds in the sense of distributions. 

The right-hand side of \eqref{eq_Lessel_sinus111} is real-valued. The right-hand side of \eqref{eq_Lessel_sinus} is real for $\vert t + k \vert < r$ and purely imaginary otherwise. Applying Lemma~\ref{th:parseval} to the integral in \eqref{eq_Lessel_sinus222} and taking the imaginary part we get 
$$
W(t,r) = \frac{1}{\sqrt{2\pi}} \left(\int\limits_{-\infty}^{-t-r} - \int\limits_{-t+r}^{\infty}\right) 
\frac{1}{\sqrt{(k+t)^2 - r^2}} \exp\left(-\frac{k^2}{2}\right) dk.
$$
By changing $k = \xi r - (t+r)$ in the first integral and $k = \xi r + (r-t)$ in the second one, we get
$$
W(t,r) = -\tilde{W}(t,r)+\tilde{W}(-t,r)
$$
with
\begin{equation}
\tilde{W}(t,r) = \frac{1}{\sqrt{2\pi}} \int\limits_{0}^{\infty}
 \frac{1}{\sqrt{\xi} \sqrt{\xi + 2}} \exp\left(-\frac{(r - t + \xi r)^2}{2}\right) d\xi.
 \label{eq_tam0_W}
\end{equation}

Taking the time and radial derivatives (see \eqref{eq_def_pur}) we come to \eqref{eq_tamnew_2r}--\eqref{eq_tamnew_3r}.
\end{proof}

To prove the last solution form, we need the Parseval identity for the Hankel transform.
\begin{lemma}
For each $f, g \in L^2(0,\infty)$ there holds
$(f, g) = (F, G)$, where $F$ and $G$ are the images of $f$ and $g$, correspondingly, under the Hankel transform.
\label{th:parseval2}
\end{lemma}

\begin{lemma}
For each $t \ge 0$, $r \ge 0$, the formulas \eqref{eq_puls_bf}--\eqref{eq_jjn} give a solution of \eqref{eq:intro:7}--\eqref{eq_def_pur}.
\end{lemma}
\begin{proof}
Rewrite \eqref{eq_expr_W} in the form
$$
W(t,r) = -\int\limits_0^\infty \omega  \left(\exp\left(-\frac{\omega^2}{2}\right)J_0(r\omega)
\right) \left(\frac{1}{\omega}  \sin(\omega t)\right) d\omega.
$$

There holds (see \cite{Gradstein1963eng}, eq. 6.633(2))
\begin{equation}
\begin{gathered}
\int\limits_0^\infty \omega   \exp\left(-\frac{\omega^2}{2}\right)
J_0(\omega r) J_0(\omega x) d\omega = 
\\
=
\exp\left(-\frac{r^2+x^2}{2}\right) 
I_0\left(\frac{2rx}{2}\right)
=
\exp\left(- \frac{(r-x)^2}{2}\right) 
\tilde{I}_0\left(rx\right).
\end{gathered}
\nonumber
\end{equation}
where \mbox{$\tilde{I}_j(z) = e^{-z} I_j(z)$} and $I_j(z)$ is the modified Bessel function of index $j$. 
Also there holds
$$
\int\limits_0^\infty \sin(\omega t) J_0(\omega x) d\omega = \frac{1}{\sqrt{(t^2 - x^2)_+}},
$$
which is a partial case of \eqref{eq_Lessel_sinus}. By Lemma~\ref{th:parseval2} we get
\begin{equation*}
W(t,r) = - \int\limits_0^{t} x \exp\left(-\frac{(r-x)^2}{2}\right) 
\tilde{I}_0\left(rx\right) \frac{1}{\sqrt{t^2 - x^2}} dx.
\end{equation*}
Changing $x = t(1-\xi)$ we get $W(t,r) = -t \J_{0,1}(t,r)$ 
where $\J_{0,1}$ is defined by \eqref{eq_jjn}. It remains to take the time and radial derivatives using that $I_0' = I_1$ and hence $\tilde{I}_0' = \tilde{I}_1 - \tilde{I}_0$.
\end{proof}

Another way to obtain \eqref{eq_puls_bf}--\eqref{eq_jjn} is by the Kirchhoff formula for the Cauchy problem for the 2D wave equation.

The form presented in Section~\ref{sect:asympt} also can be considered as the application of the Parseval identity, with respect to the series with the Hermite polynomials.

\section{Accuracy of the numerical integration}
\label{sect:accuracy}

In this section we study the accuracy of the quadrature formulas in use. Thus we prove that our algorithm gives the solution with the precision $\varepsilon$.

\subsection{General estimates}

\begin{lemma}\label{th:accuracy1}
Let $0 < h \le \pi$, $n \in \mathbb{N}$. Denote $L = (n+1/2)h$.
Let $f(z)$ be holomorphic at $\vert \mathrm{Re}z\vert  \le L$, $\vert \mathrm{Im}z\vert  \le 2\pi/h$. 
Consider the integral
\begin{equation*}
I = \int\limits_{-\infty}^{\infty} f(x) \exp(-x^2/2) dx
\end{equation*}
and its $(2n+1)$-point approximation
$$
Q = h\sum\limits_{k=-n}^{n} \exp\left(-\frac{(kh)^2}{2}\right) f(kh).
$$
Then there holds
\begin{equation*}
\begin{gathered}
\vert  I - Q\vert  \le \left \vert \left(\int\limits_{-\infty}^{-L}+\int\limits_{L}^{\infty}\right) f(x) \exp(-x^2/2) dx\right\vert
+
\\
+ \frac{4h}{\pi}\exp\left(-\frac{L^2}{2}\right) f_1 + 5.2\exp\left(-\frac{2\pi^2}{h^2}\right) f_0
\end{gathered}
\end{equation*}
where
\begin{equation}
f_0 = \max\limits_{\vert \mathrm{Re}z\vert  \le L, \ \vert \mathrm{Im}z\vert  = 2\pi/h} \vert f(z)\vert , \quad
f_1 = \max\limits_{\vert \mathrm{Re}z\vert  = L, \ \vert \mathrm{Im}z\vert  \le 2\pi/h} \vert f(z)\vert .
\label{eq_def_f0f1}
\end{equation}
\end{lemma}

The proof of Lemma~\ref{th:accuracy1} is given in Appendix. It follows \cite{Goodwin1949}, where a similar result was proved for $n = \infty$ and fixed $h$.

Now let $w\in L^1(-1,1)$ be positive a. e. in $(-1,1)$. 
Let $f \in C^2[-1,1]$. Consider the integral
$$
I(f) = \int\limits_{-1}^1 f(x) w(x) dx
$$
and its $m$-point Gaussian quadrature approximation $Q_m(f)$. 

Let $T_n(x) = \cos(n \arccos x)$, $x \in [-1,1]$, be the Chebyshev polynomials. Then there holds
\begin{equation*}
f(x) = \frac{a_0}{2} + \Bigl.\sum\limits_{n=1}^{\infty}\Bigr. a_n T_n(x).
\end{equation*}
The coefficients $a_n$ are given by 
$$
a_n = \frac{2}{\pi}\int\limits_{-1}^1 \frac{f(x) T_n(x)}{\sqrt{1-x^2}} dx, \quad n \in \mathbb{N} \cup \{0\}.
$$
Since Gaussian quadratures are exact for $f = T_n$, $n \le 2m-1$, we get
\begin{equation}
\vert Q_m(f) - I(f) \vert = \left\vert \sum\limits_{k=m}^{\infty} a_{2k} (Q_m(T_{2k}) - I(T_{2k})) \right\vert
\le
2\mu_0 \sum\limits_{k=m}^{\infty} \vert a_{2k} \vert
\label{eq_err_preestimate}
\end{equation}
where
\begin{equation*}
\mu_0 = \int\limits_{-1}^1 w(x)dx
\end{equation*}
(for a detailed proof see formulas 3.17 and 3.18 in \cite{Hunter1995}).

For $a_n$, we use estimates based on an elliptic contour, which are obtained in \cite{Eliott1966}. That paper does not state bounds for $a_n$; instead, it gives approximate values $\tilde{a}_n$ such that $\tilde{a}_n / a_n \rightarrow 1$ as $n \rightarrow \infty$.  Together with \eqref{eq_err_preestimate}, these results immediately yield the accuracy estimates formulated below as Lemma~\ref{th:accuracy2} and Lemma~\ref{th:accuracy3}.

Let $\rho = 1+\sqrt{2}$ and $\Gamma$ be the ellipse with the focal points $z = \pm 1$ and semi-axes $\sqrt{2}$ and $1$.

\begin{lemma}\label{th:accuracy2}
Let $f(z)$ be holomorphic inside $\Gamma$.
Then there holds
\begin{equation*}
\vert Q_m(f) - I(f) \vert \le 4 C_a \mu_0 \frac{\max\limits_{\Gamma} \vert f(z)\vert}{{\rho^{2m}}} \frac{\rho}{\rho-1}
\end{equation*}
where $C_a \rightarrow 1$ as $m \rightarrow \infty$.
\end{lemma}


\begin{lemma}\label{th:accuracy3}
Let $c \in \mathbb{R} \setminus [-1,1]$, $f(z) = g(z)/\sqrt{z-c}$, $g(c) \ne 0$, and $g(z)$ be holomorphic inside $\Gamma$. Then there holds
\begin{equation}
\begin{gathered}
\vert Q_m(f) - I(f) \vert \le 4 C_a \mu_0 \frac{\max\limits_{\Gamma} \vert f(z)\vert}{{\rho^{2m}}} \frac{\rho}{\rho-1} + \\
+ 2 C_a \delta \mu_0   \frac{(c^2-1)^{1/4}}{\sqrt{2\pi m}(\vert c \vert + \sqrt{c^2-1})^{2m-2}(\vert c \vert - 1)} \vert g(c) \vert,
\end{gathered}
\label{eq_Em2}
\end{equation}
where $\delta = 1$ if $c$ is inside $\Gamma$ and $\delta = 0$ otherwise,
and $C_a \rightarrow 1$ as $m \rightarrow \infty$.
\end{lemma}

For a review of accuracy estimates of Gaussian quadratures see \cite{Gautschi1981}.

\subsection{The case $r > t+H$}

In this case, the signal from the circle of size $H$ have not reached the point $(t,r)$. The solution for pressure and velocity pulsations are less than $\varepsilon$ and may be approximated by zero.

\subsection{Approximation of \eqref{eq_form1}. The case $t+r \le 1.05H$}
\label{sect:appr1}

Since 
$$
\left\vert \int\limits_H^{\infty} \omega e^{-\omega^2/2} J_0(r\omega) \cos(t\omega) d\omega \right\vert < \varepsilon/2,
\quad
\left\vert \int\limits_H^{\infty} \omega e^{-\omega^2/2} J_1(r\omega) \sin(t\omega) d\omega \right\vert < \varepsilon/2,
$$
we crop the integration range to $(0,H)$. The variable change $\omega = H(1+\xi)/2$ in \eqref{eq_form1} yields
$$
p' \approx I(f_p) = \int\limits_{-1}^{1} f_p(t,r,\xi) d\xi, \quad 
u'_r \approx I(f_u) = \int\limits_{-1}^{1} f_u(t,r,\xi) d\xi
$$
with
$$
\left(\begin{array}{c}f_p(t,r,\xi) \\ f_u(t,r,\xi) \end{array}\right) =  \frac{H^2}{4} (1+\xi) e^{-\frac{H^2(1+\xi)^2}{8}} 
\left(\begin{array}{c}J_0(rH(1+\xi)/2) \cos(tH(1+\xi)/2) \\ J_1(rH(1+\xi)/2) \sin(tH(1+\xi)/2) \end{array}\right).
$$

We evaluate this integral using the Gauss~-- Legendre quadrature. Its accuracy can be estimated by Lemma~\ref{th:accuracy2}. On $\Gamma$, there holds \mbox{$\vert\mathrm{Im}\xi\vert \le 1$}. Since
\begin{equation}
\vert\exp(-z^2)\vert = \exp((\mathrm{Im} z)^2 - (\mathrm{Re} z)^2), \quad z \in \mathbb{C},
\label{eq_expz2}
\end{equation}
and $\vert J_j(z)\vert \le \exp(\vert\mathrm{Im}z\vert)$, $\vert\sin z\vert \le \exp(\vert\mathrm{Im}z\vert)$, $\vert\cos z \vert \le \exp(\vert\mathrm{Im}z\vert)$, then 
$$
\max\{f_p, f_u\} \le \frac{H^2}{2} \exp\left(\frac{H^2}{8} + \frac{(r+t) H}{2}\right), \quad \xi \in \Gamma.
$$
Recall that we use the form \eqref{eq_form1} for $t+r \le 1.05H$. Then by Lemma~\ref{th:accuracy2} we get
$$
\vert Q_m(f_p) - I(f_p) \vert \le 2(2+\sqrt{2}) C_a H^2 \exp\left(\frac{5.2 H^2}{8}\right) (1+\sqrt{2})^{-2m}.
$$
Taking $m=M_3 =\lceil 0.71 H^2 \rceil +1$ we get $\vert Q_m(f_p) - I(f_p) \vert  \le C_a\varepsilon/2$, and the same holds for the velocity pulsation.

Alternatively, for the form \eqref{eq_form1} one may use Gauss~-- Hermite quadratures, but for $t+r \approx H$ the required number of nodes is approximately the same as for the Gauss~-- Legendre rules.

\subsection{Appoximation of \eqref{eq_tamnew_2r}--\eqref{eq_tamnew_3r}}

\subsubsection{The case $t>r+1.152H$, $r>R_1$}\label{sect:341} 

By the variable change $r - t + r\xi = \eta$, the integral \eqref{eq_tamnew_3r} transforms to 
$$
\mathcal{J}_j(t,r) = \frac{1}{\sqrt{2\pi} r} \int\limits_{-\infty}^{\infty} \exp(-\eta^2/2) f_j(\eta) d\eta, \quad f_j(\eta) = \eta \frac{(1+\xi)^j}{\sqrt{\xi(\xi+2)}} \Theta(\xi),
$$
where $\xi \equiv \xi(\eta) = (t + \eta)/r - 1$ and $\Theta$ is the Heaviside function. Note that $f_j(0) = 0$. Consider the approximation
\begin{equation}
\begin{gathered}
Q_j(t,r) =\frac{1}{\sqrt{2\pi} r} h\sum\limits_{k=1}^{M_2} \exp\left(-\frac{(kh)^2}{2}\right) (f_j(kh) + f_j(-kh))
\end{gathered}
\label{eq_equidistant_formula}
\end{equation}
where $M_2 = \lceil 0.2 H^2 \rceil$, $h = \sqrt{2\pi/(M_2+1/2)}$, and $L = \sqrt{2\pi(M_2+1/2)}$. 
Note the possibility that
$$
\vert f_j(kh) + f_j(-kh) \vert \ll \vert f_j(kh) \vert + \vert f_j(-kh) \vert,
$$
so instead of the direct evaluation of $f_j(kh) + f_j(-kh)$ we use its regularized form.

Now we estimate the accuracy of \eqref{eq_equidistant_formula}.
By Lemma~\ref{th:accuracy1} we get
\begin{equation}
\vert \J_j - Q_j\vert \le  \frac{1}{\sqrt{2\pi} r} (E_1 + E_2 + E_3)
\label{eq_lemma_equdistant_r}
\end{equation}
with
$$
E_1 = \left\vert\left(\int\limits_{-\infty}^{-L}+\int\limits_{L}^{\infty}\right) f_j(\eta) \exp(-\eta^2/2) d\eta\right\vert,
$$
$$
E_2 = \frac{4h}{\pi}\exp\left(-\frac{L^2}{2}\right) f_{j,1},
\quad 
f_{j,1} = \max\limits_{\vert \mathrm{Re}z\vert  = L, \ \vert \mathrm{Im}z\vert  \le 2\pi/h} \vert f_j(z)\vert,
$$
$$
E_3 = 5.2 \exp\left(-\frac{2\pi^2}{h^2}\right) f_{j,0},
\quad
f_{j,0} = \max\limits_{\vert \mathrm{Re}z\vert  \le L, \ \vert \mathrm{Im}z\vert  = 2\pi/h} \vert f_j(z)\vert.
$$

To estimate $E_1$, note that $\vert f_j(\eta) \vert \le \vert \eta \vert (1 + \xi^{-1/2}) \Theta(\xi)$ for $j = 0,1$. Then
$$
\int\limits_{\hat{H}}^{\infty} \vert f_j(\eta) \vert \exp(-\eta^2/2) d\eta
\le
\int\limits_{\hat{H}}^{\infty} \eta \exp(-\eta^2/2) \left(1 + \frac{\sqrt{r}}{\sqrt{\eta + t-r}}\right) d\eta 
\le
$$
$$
\le
\exp(-L^2/2) \left(1 + \frac{\sqrt{r}}{\sqrt{2.152H}} \right).
$$
Let $\xi_* = \xi(-L) = (t-L-r)/r$, then
$$
\int\limits_{-\infty}^{-L} \vert f_j(\eta) \vert \exp(-\eta^2/2) d\eta 
\le
L \exp(-L^2/2) \frac{r}{t} \int\limits_{0}^{\xi_*} (1 + \xi^{-1/2}) d\xi =
$$
$$
= L \exp(-L^2/2) \frac{r}{t} \left(\xi_* + \sqrt{2\xi_*}\right)
\le 
L \exp(-L^2/2) (1+\sqrt{2}).
$$
Thus
$$
r^{-1} E_1 \le 
R_1^{-1} \exp(-L^2/2) (1 + \sqrt{R_1/(2.152H)} + L(1+\sqrt{2})).
$$
One may check that the expression on the right-hand side is less than $0.6 \varepsilon$ provided that $\varepsilon \le 2\cdot 10^{-16}$ (in particular, for $\varepsilon = 2\cdot 10^{-16}$ it is approximately equal to $0.56 \varepsilon$).

Now we obtain estimates for $E_2$ and $E_3$. For $\mathrm{Re}\,\xi \ge \delta > 0$ there holds
$$
\left\vert\frac{1}{\sqrt{\xi(\xi+2)}}\right\vert \le \frac{1}{\vert\xi\vert} \le \frac{1}{\delta}, \quad 
\left\vert\frac{1+\xi}{\sqrt{\xi(\xi+2)}}\right\vert \le 1 + \frac{1}{\delta}.
$$
In our case, $\delta = \xi(-L) = (t-L-r)/r$. Then for $j=0,1$ and $l=0,1$ we have
$$
f_{j,l} \le \left(1+\frac{r}{t-L-r}\right) \left(L^2 + \left(\frac{2\pi}{h}\right)^2\right)^{1/2} \le
\sqrt{2} (1+4r) L.
$$
Here we used $h = 2\pi/L$ and $t > r + L + 0.25$. Then
$$
E_2 + E_3 \le \sqrt{2}(8 + 5.2L) \exp\left(-\frac{L^2}{2}\right) (R_1^{-1} + 4) r.
$$
For $\varepsilon \le 2 \cdot 10^{-16}$, we have $L \ge 1.13H$. From here it may be shown that $(E_2 + E_3)/r \le \varepsilon/4$, and finally \eqref{eq_lemma_equdistant_r} yields $\vert Q_j - \J_j\vert \le \varepsilon/2$.

\subsubsection{The case $r-H < t < r+1.152 H$, $t+r > 1.05H$, $r \ge R_2$} 

The integrand in \eqref{eq_tamnew_3r} is not positive, so we need to be aware the accumulation of the arithmetic errors when proceeding with the numerical integration. For $j=0$, our numerical experiments do not show this effect. However, for $j=1$, this may be a problem. So instead of  \eqref{eq_tamnew_3r}, we use the form
\begin{equation}
\J_j(t,r) = \frac{1}{\sqrt{2\pi}} \int\limits_0^{\infty} \frac{\exp(-(r-t+r\xi)^2/2)}{\sqrt{\xi(\xi+2)}} \left(\frac{r(1+\xi)-t}{(1+\xi)^j} + \frac{j}{r(1+\xi)^2}\right) d\xi.
\label{eq_tamnew_3r_reg}
\end{equation}
The prove the equivalence of \eqref{eq_tamnew_3r} and \eqref{eq_tamnew_3r_reg} for $j=1$, one needs to split \eqref{eq_tamnew_3r_reg} into two integrals and integrate the second one by parts using
$$
\int \frac{1}{(1+\xi)^2 \sqrt{\xi(\xi+2)}}d\xi = \frac{\sqrt{\xi(\xi+2)}}{\xi+1} + c.
$$

To evaluate \eqref{eq_tamnew_3r_reg} we crop the integration range to $(0,b)$ with 
$$
b = \frac{t+H}{r}-1.
$$
By the variable change $\xi(\eta) = b (\eta+1)/2$, the integral transforms to 
$$
\J_j(t,r) \approx \int\limits_{-1}^1 w(\eta) f_j(\eta)d\eta, \quad f_j(\eta) = \frac{g_j(\eta)}{\sqrt{\eta - c}}
$$
with $w(\eta) = (\eta+1)^{-1/2}$, $c = -1 - 4/b$,
$$
g_j(\eta) = \frac{1}{\sqrt{2\pi}} \exp(-(r-t+r\xi)^2/2)\left(\frac{r(1+\xi)-t}{(1+\xi)^j} + \frac{j}{r(1+\xi)^2}\right), \quad \xi \equiv \xi(\eta).
$$

To evaluate $\J_j(t,r)$, we use the Gauss~-- Jacobi quadrature rule with $M_3 = \lceil 0.71 H^2 \rceil + 1$ nodes. For the error estimate, we use Lemma~\ref{th:accuracy3}. Assume that $\varepsilon \le 2 \cdot 10^{-16}$. We have $\xi(c) = -2$ and thus
$$
\vert g_j(c)\vert = \frac{1}{\sqrt{2\pi}} \exp(-(t+r)^2/2) (t+r+r^{-1}).
$$
Since $t+r > 1.05H$ and $r \ge R_2 = 5\varepsilon^{1/10}$, we have $\vert g_j(c)\vert < \varepsilon/10$. Also $b \le 2.152H/r$ and
$$
c \le -1 - \frac{20 \varepsilon^{1/10}}{2.152 H}.
$$
From here, it can be shown that the second term in \eqref{eq_Em2} does not exceed $\varepsilon/4$.

Now estimate the maximum of $\vert g_j(z)\vert$ on $\Gamma$. From  \eqref{eq_expz2} on $\Gamma$ we get
$$
\vert g(z)\vert \le \frac{C}{\sqrt{2\pi}} \exp\left(\frac{r^2 b^2}{8}\right) \exp(-(r-t+r\mathrm{Re}z)^2/2) (\vert r-t+rz\vert\  + r^{-1})
$$
where $C = \max_{\Gamma} \max\{\vert 1 + z \vert, \vert 1 + z \vert^{-2}\} < 6$. Hence,
$$
\vert g(z)\vert \le \frac{6}{\sqrt{2\pi}}  \exp(H^2/2) (1 + R_2^{-1}).
$$
Therefore, by \eqref{eq_Em2},
$$
\vert Q_j - \J_j\vert \le 12 C_a \exp(H^2/2) (1+\sqrt{2})^{-2M_3} + \frac{\varepsilon}{2} \le (1/2+C_a/4)\varepsilon.
$$

\subsection{Approximation of \eqref{eq_puls_bf}--\eqref{eq_jjn}}

The form \eqref{eq_puls_bf}--\eqref{eq_jjn} is used for a small domain of parameters, which is a subset of the rectangle $1.05H-R_2 \le t \le 1.3H$ and $0 \le r \le R_2$, see Fig.~\ref{fig:variants}.

We crop the integration range to $(a,1)$ with
$$
a = 1 - \frac{r+H}{t} > 0,
$$
so 
$$
\J_{j,n}(t,r) \approx \int\limits_a^1 \exp\left(- \frac{(r-t+t\xi)^2}{2}\right) \tilde{I}_j(rt(1-\xi)) \frac{(1-\xi)^n}{\sqrt{\xi(2-\xi)}} d\xi.
$$
Here $\tilde{I}_j(x) = e^{-x} I_j(x)$, and $I_j(x)$  is the modified Bessel function. After scaling
$$
\xi \equiv \xi(\eta) = \frac{1+{a}}{2} + \frac{1-{a}}{2} \eta
$$
we get
\begin{equation}
\J_{j,n}(t,r) \approx 
\int\limits_{-1}^1 \frac{g_{j,n}(\eta)}{\sqrt{\eta-c}}d\eta,
\label{eq_aux_13pj}
\end{equation}
with $c = - (1+a)/(1-a) = 1 - 2t/(r+H)$, $g_{j,n}(\eta) = \hat{g}_{j,n}(1-\xi(\eta))$, 
$$
\hat{g}_{j,n}(\zeta) =   \left(\frac{1-a}{2}\right)^{1/2}  \frac{\zeta^n }{\sqrt{1+\zeta}} h_j(\zeta), \quad
h_j(\zeta) = \exp\left(- \frac{(r-t\zeta)^2}{2}\right) \tilde{I}_j(rt\zeta).
$$
For $t$ and $r$ that are the case and $\varepsilon < 2\cdot 10^{-16}$, we have $c < -1.0398$.

Let $Q_{j,n}$ be the Gauss~-- Legendre quadrature for \eqref{eq_aux_13pj} with $M_3 = \lceil 0.71 H^2 \rceil$ nodes. To estimate its accuracy, we use Lemma~\ref{th:accuracy3}. 

Using
$$
\tilde{I}_j(z) \equiv e^{-z} I_j(z) = \frac{1}{\pi}\int\limits_0^{\pi} \exp(z(\cos(\phi)-1)) \cos(j\phi) d\phi
$$
we get $\vert\tilde{I}_j(z) \vert \le  \exp(-\mathrm{Re}\,z)$ for $\mathrm{Re}\,z < 0$ and $\vert\tilde{I}_j(z) \vert \le 1$ for \mbox{$\mathrm{Re}\,z \ge 0$}. 
Using \eqref{eq_expz2} we get
\begin{equation*}
\begin{gathered}
\vert h_j(\zeta) \vert \le 
\exp\left(- \frac{(r-t\mathrm{Re}\zeta))^2 - (t\mathrm{Im}\zeta)^2}{2}\right) \max\{1, \exp(-rt\mathrm{Re}\zeta))\} =
\\
= 
\max\left\{
\exp\left(- \frac{(r-t\mathrm{Re}\zeta)^2 - (t\mathrm{Im}\zeta)^2}{2}\right),
\exp\left(- \frac{(r+t\mathrm{Re}\zeta)^2 - (t\mathrm{Im}\zeta)^2}{2}\right)
\right\}.
\end{gathered}
\end{equation*}
Thus for each $t,r \ge 0$ and $\zeta \in \mathbb{C}$ there holds $\vert h_j(\zeta)\vert \le \exp((t \mathrm{Im}\zeta)^2/2)$. 

We have
$$
\max\limits_{\eta \in \Gamma} \vert t\,\mathrm{Im}\xi(\eta)\vert = t\frac{1-{a}}{2} \max\limits_{\Gamma} \vert\mathrm{Im}\eta\vert = t\frac{1-{a}}{2}
= \frac{r+H}{2}.
$$
Hence, taking into account that $0 \le n \le 2$,
\begin{equation*}
\begin{gathered}
\max\limits_{\eta \in \Gamma} \vert g_{j,n}(\eta) \vert \le \left(\frac{1-a}{2}\right)^{1/2}  \max\limits_{\eta \in \Gamma} \frac{(1-\xi(\eta))^n }{\sqrt{2-\xi(\eta)}} \exp((t(1-a)/2)^2/2) \le
\\
\le 13 \exp\left(\frac{(R_2+H)^2}{8}\right).
\end{gathered}
\end{equation*}

As in the previous case, by Lemma~\ref{th:accuracy3} we get
$$
\vert Q_{j,n} - \J_{j,n}\vert \le (1/2+C_a/4)\varepsilon.
$$

\section{Verification}
\label{sect:experiment}

In this section we measure the actual accuracy of our algorithm in floating-point arithmetics. For the double-double and quad-double precision arithmetics, we use the QD library \cite{lib:QD}.

First we check the validity of the algorithm. We put $\varepsilon = 2\cdot 10^{-16}$ and use the double-double precision arithmetics (i. e. $10^{-32}$ relative accuracy). We calculate the values of $p'(t,r)$ and $u'(t,r)$ on the lattice $t = 1.01^n, r = 1.01^m$ with $n, m = -1000, \ldots, 1000$. The reference solution was calculated with the same arithmetic model and $\varepsilon = 4 \cdot 10^{-32}$. The maximal difference from the reference solution was $1.97 \times 10^{-16} \le \varepsilon$. It primarily originates from the cropping of the integration range.

Now we consider the influence of the floating-point arithmetics. We use $\varepsilon = 2\cdot 10^{-16}$ for the double-precision arithmetics and $\varepsilon = 4 \times 10^{-32}$ for the double-double precision one. The reference solutions was obtained by the same algorithm with and the quad-double precision arithmetics and $\varepsilon = 8 \times 10^{-64}$. The error of the numerical solution did not exceed $2.09 \times 10^{-15}$ for each variable in the double precision and $1.42 \times 10^{-30}$ in the double-double precision. The maximal error appears when the form \eqref{eq_form1} is used for $t+r \approx H$, and this is because the integrand changes its sign several times within the integration range.

The total time to compute the solution for 4 million points on Intel Core i7-10700 CPU core was 3.98 seconds in the double precision and 405 seconds in the double-double precision. The most time consuming approximation is the one based on the form \eqref{eq_form1} because of the Bessel function evaluations.

\appendix
\section*{Appendix A. Proof of Lemma~\ref{th:accuracy1}}

Put
$$
g(z) = (f(z) + f(-z))/2.
$$
By symmetry,
$$
Q = h\sum\limits_{k=-n}^{n} \exp\left(-\frac{(kh)^2}{2}\right) g(kh).
$$
Also put
$$
J = \int\limits_{-L}^{L} f(x) \exp(-x^2/2) dx.
$$

Let $\Gamma$ be the counterclockwise oriented rectangular contour with the corners $\pm L \pm 2\pi i/h$. Let $\Gamma_+ = \Gamma \cap \{\mathrm{Im}z > 0\}$ and $\Gamma_- = \Gamma \cap \{\mathrm{Im}z < 0\}$.
Then by the Cauchy residue theorem
$$
Q = \oint\limits_{\Gamma} g(z) \frac{\exp(-z^2/2)}{1 - \exp(-2\pi iz/h)} dz = 
\left(\int\limits_{\Gamma_+} + \int\limits_{\Gamma_-}\right)
g(z) \frac{\exp(-z^2/2)}{1 - \exp(-2\pi iz/h)} dz.
$$

Note that
$$
\left(\int\limits_{\Gamma_+} - \int\limits_{\Gamma_-}\right) g(z) \frac{\exp(-z^2/2)}{1 - \exp(-2\pi iz/h)} dz = 
\int\limits_{\Gamma_+} g(z) \exp(-z^2/2) dz = -J.
$$
The first identity in this chain is by the substitution $-z$ for $z$, with the use of 
$$
\frac{1}{1-\exp(z)} + \frac{1}{1-\exp(-z)} = 1.
$$
Thus, the integration error is 
$$
Q - J = \int\limits_{\Gamma_+} G(z) dz, \quad G(z) = 2g(z) \frac{\exp(-z^2/2)}{1 - \exp(-2\pi iz/h)}.
$$

Let $\gamma_+$ and $\gamma_-$ be the vertical segments of $\Gamma_+$, and $\gamma_0$ be its horizontal segment. 

First consider the integral over $\gamma_0$. Here $z = x + 2\pi i/h$ and
$$
1 - \exp(-2\pi iz/h) = 1 -  \exp(-2\pi i x/h)\exp(4\pi^2/h^2).
$$
Then
$$
\left\vert \int\limits_{\gamma_{0}} G(z) dz\right\vert \le  
2f_0 
\int\limits_{-L}^{L} \left\vert \frac{\exp(-x^2/2)\exp(2\pi^2/h^2)}{\exp(4\pi^2/h^2)-1} \right\vert dx
\le  
$$
$$
\le
2 \sqrt{2\pi} \frac{1}{1 - \exp(-4\pi^2/h^2)} f_0 \exp\left(-2\frac{\pi^2}{h^2}\right)
\le 5.2 f_0 \exp\left(-2\frac{\pi^2}{h^2}\right). 
$$
Here we used $h \le \pi$, which is by assumption.

On vertical segments we have $z = \pm L + iy$, $0 \le y \le 2\pi/h$, thus
$$
1 - \exp(-2\pi iz/h) = 1 - \exp(\mp 2\pi i L/h)\exp(2\pi y/h) = 
$$
$$
=
1 - \exp(\mp 2\pi i (n+1/2))\exp(2\pi y/h) =
1 + \exp(2\pi y/h)
$$
and
$$
\left\vert \int\limits_{\gamma_{\pm}}G(z) dz\right\vert \le  
2f_1 \exp(-L^2/2)
\int\limits_0^{2\pi/h}  \frac{\exp(y^2/2)}{1 + \exp(2\pi y/h)} dy.
$$
Using $y^2 \le (2\pi/h) y$ and dropping unit in the denominator we get
$$
\int\limits_0^{2\pi/h} \frac{ \exp(y^2/2)}{1 + \exp(2\pi y/h)} dy
\le  
\int\limits_0^{\infty} \exp\left(-\frac{\pi y}{h}\right) dy = \frac{h}{\pi}.
$$
Thus
$$
\left\vert \int\limits_{\gamma_{\pm}} G(z) dz\right\vert \le \frac{2h}{\pi} \exp(-L^2/2) f_1.
$$

Combining the estimates and using $\vert I - Q \vert \le \vert I - J \vert + \vert J - Q \vert$ we get the statement of the lemma.

\end{document}